\documentclass{amsart}
\usepackage[utf8]{inputenc}
\usepackage[french]{babel} 
\usepackage{mathptmx}
\usepackage{graphicx}
\usepackage{helvet}

\newtheorem{theorem}{Théorème}
\newtheorem{proposition}{Proposition}

\author{Alessandra Cauli} 

\address{Dipartimento di Scienze Matematiche, Politecnico di Torino, Corso Duca degli Abruzzi 24, 10129 Torino, Italy}
\email{alessandra.cauli@unito.it}

\title{Sur la résolution des équations de troisième et quatrième degré}

\date{}

\usepackage[T1]{fontenc}

\begin{document}

\maketitle

\section{Abstract}
This paper is the authored translation in French language which it is its second edition.
\vspace{10pt} \\
\textbf{Keywords:} algebraic equations; radicals; the third and fourth degree; History of \\ Mathematics; Galois Theory.

\section{Résumé}

Cet article encadre la résolution des équations de troisième et quatrième degré par radicaux. On présente certain détails historiques sur ce problème fondamental. En outre, il s'agit de méthodes pratiques pour la résolution des équations de troisième et quatrième degré par le point de vue algébrique et nous allons introduire des résultats sur les équations de majeur degré. Les Mathématiques et les nombres complexes et les résultats des règles des calculassions aussi bien d'avoir complété la théorie des équations entaillent des modèles mathématiques en différents contextes historiques et chronologiques au plusieurs niveaux de recherches. Le sujet rejoindrait un auditorium de référence pour les auteurs, enseigneurs et étudiants au moment où le contexte de collocation avantage les équations algébriques et les solutions par les radicaux.
\vspace{10pt} \\
\textbf{Clés:} équations algébriques; radicaux; troisième et quatrième degré; histoire des \\ mathématiciens; théorie de Galois.

\section{Introduction}

L'étude des équations algébriques est trés fondamental et c'est un concept basilaire des Mathématiques autant que tous les problèmes seront irrésolubles pendant la vie. John L. Berggren confirme en l'Encyclopædia Britannica que "l'algèbre élémentaire est trés importante surtout pour les Mathématiciens et pour les sciences naturelles, les sciences computationnelles, l'Ingénierie, l'Economie et les Affaires Financiéres. Autrement dit que l'écriture, l'algèbre élémentaire est un trait d'union entre la Civilisation et la Pensée Scientifique. Jeuneuses peuples de civilisations des Babyloniens, des Illinois, des Indes, Chinois et Islamiques, offriront le développement de l'algèbre élémentaire pour un système établie, stable et indépendant pour la réalisation et la représentation des nombres réels et un symbolisme pour la considération des liens entre ces nombres et une opération". Les équations algébriques sont trés fondamentales pour la résolution des problèmes quotidiens et des modèles mathématiques. Le problème de la résolution des équations algébriques s'approche par le point de vue de l'Histoire des Mathématiques aux mathématiciens pour les siècles afin de trouver des solutions aux problèmes les plus énormes, surtout pour les équations de majeur degré. On retrouvera des résolutions avec formules et méthodes qui peuvent improuver les arguments et l'analyse des différents contextes historiques et chronologiques pendant lesquels les grands mathématiciens ont étés impliqués sur la résolution des équations algébriques. Nous ferons référence à \cite{Boyer} pour tous les détails historiques que nous présenterons. Pour ce qui concerne la résolution des équations algébriques en général, le théorème fondamental de l'algèbre nous afferme qu'il existe une solution réelle ou complexe d'une équation algébrique $f(z)=0$ de degré $n\geq 1$ en la seule variable $z$ avec des coefficients complexes ou réels en particulier. Le théorème fondamental de l'algèbre a été démontré pour la première fois en 1799 par le mathématicien allemand Carl Friedrich Gauss en sa thèse de Doctorat. En 1806, Jean-Robert Argand publiait une démonstration de ce théorème pour les cas des coefficients complexes. Aprés, en 1816, Gauss publiait deux autres démonstrations, une autre en 1850. Avec ces contrebutes, des autres démonstrations seront proposées. Nous connaissons les formules \[x=-\frac{b}{a}\] et \[x=-b\pm \frac{\sqrt{b^{2}-4ac}}{2a}\] qui donnent les solutions de l'équation algébrique des premier et deuxième degrés, respectivement de $ax+b=0$ et $ax^2+bx+c=0$. Une première épreuve de résolution était depuis quatre-mille ans par moyen des tablettes d'argile des Babyloniens quand ils utilisaient la méthode de la "complétion des carrés" pour la résolution des équations de deuxième degré. En effet, il semblait qu'ils voulaient reconduire les équations de deuxiéme degré au problème typique de trouver deux nombres dont la somme et le produit sont connus. Ces nombres sont exactement les deux solutions de l'équation de deuxième degré écrite en la forme $x^2-(somme)x+(produit)=0$. Cette méthode était probablement déjà connue à Héron D'Alexandrie ($\sim$ 10-70 d.C.) et Diophante D'Alexandrie (né autour les 201 et 215 d.C., mort à l'age de quatre-vingt-quatre ans, probablement autour les 285 et 299 d.C.) qui introduisaient un symbolisme pour ces arguments. La caractérisation indispensable de l'algèbre est l'utilisation des simples symboles pour représenter des quantités numériques et des opérations mathématiques. Aujourd'hui, les premières lettres de l'alphabet $(a,b,c,...)$ représentent typiquement les données, mais arbitrairement les nombres d'un problème et les lettres à la fin de l'alphabet, en particulier $x, y$ et $z$ représentent des quantités ou des variables. La représentation géométrique des quantités algébriques a énormément développé et amplis l'environnement des applications algébriques. 
Ce manuscrit est organisé à travers cinq sections: aprés une brève introduction et un flashback sur les équations de deuxième degré, la deuxième partie traite des résolutions des équations de troisième degré et de leur histoire, la troisième est dédiée à la résolution des équations de quatrième degré et de leur histoire, la quatrième contient des notions de la théorie de Galois sur la résolution des équations algébriques par moyen des radicaux, nous allons conclure avec des résultats que nous avons sur les équations de majeur degré.

\section{Résolution des équations de troisième degré}

\subsection{Histoire de l'invention de la formule de solution}

L'équation cubique $x^3+ax^2+bx+c=0$, à exception de cases spéciales, a mis au défi les mathématiciens lorsque à la période des Babyloniens, depuis quatre-mille ans, jusqu'au moment où Luca Pacioli (1445-1514) écriait en la Summa, en 1494, son hypothèse que l'équation cubique était impossible. Plusieurs détails sur la vie et les oeuvres de Luca Pacioli on peut les retrouver en \cite{Sangster}. En ce qui concerne l'histoire de la solution des équations de troisième degré, la plupart des mathématiciens grecs et arabes était souvent convolés d'après Archiméde de Syracuse, mais ils cherchaient à résoudre uniquement des cas particulières, sans réussir en une méthode générale. Le contrebute du mathématicien Perse Omar Khayyam était trés importante (1100 ca.), on peut retrouver plusieurs détails en \cite{Cooper}. Il élaborait l'oeuvre ``L'algèbre'' pour décrire une méthode générale et reconnaitre quand les équations de troisième degré aient des racines positives en donnant une classification des équations en treize cas à exception des coefficients négatifs. Egalement, la résolution des équations cubiques généralisait la méthode des intersections coniques. Nous présentons la procédure d'Omar Kayyam en utilisant des modernes concepts et leurs dénotations. Considérant l'équation $x^3+ax^2+bx+c=0$, on prend $x^2=2py$ et on obtient $2pxy+2apy+b^2$ $x+c^3=0$ qui représente une hyperbole tandis que la précédente est une parabole. Si les deux courbes seront dessinées dans un même système des coordonnées de référence, les points d'intersection sont les racines de l'équation considérée. Elles peuvent être utilisées également pour d'autres sections coniques. Il était Scipione del Ferro (1465-1526), un professeur des mathématiques à Boulogne réussit à résoudre l'équation cubique $x^3+px=q$ autour 1515. En effet, il ne publiait pas la méthode de résolution parce que pendant celle période les découvertes furent tenues cachées et utilisée pour défier les duelles à résoudre les mêmes problèmes. Scipione Del Ferro révélait la méthode à la fin de sa vie pour un étudiant,  Antonio Maria Fior. Les nouvelles commencent à se diffuser et porteront Niccolò Fontana de Bréche (1499-1559), autrement dit Tartaglia pour un mot tombé en bataille, en 1512, pour la défense de Bréche de la part de la France, pour la retrouvée de la solution en 1530: il résoudrait les équations de type $x^3+px=q$ et $x^3+px^2=q$ avec $p$ et $q$ positifs. Il dit d'avoir résolu le problème mais la formule restait secrète. Fior douillait Tartaglia en 1535 et chacun des défiants proposaient trente problèmes. Les nouvelles de l'exode de Tartaglia en la défie rejoignaient Girolamo Cardano (1501-1576), un physique à la Courte de Milan, philosophe, astrologue et mathématicien. Tartaglia, aprés l'insistance de Cardano, finalement lui révélait la méthode de la résolution, mais Cardano promisait de tenir en secret celle méthode. En effet, Tartaglia, au but d'écrire la formule, donnait à Cardano la solution en la forme d'une poésie. Initialement, Cardano ne comprenait pas et questionnait l'aide de Tartaglia qui demandait une explication avec plusieurs détails. Pour cette raison, Cardano assemble son pupille Ludovico Ferrari (1522-1565), commenéait à travailler sur les équations de troisième degré en trouvant donc une démonstration de la solution. Ferrari découvrait en outre la solution de l'équation de quatrième degré qui lui amenait dans le monde des grands mathématiciens. Il y avait un problème: un passage de la solution regardait la formule de la résolution des équations de troisième degré que Cardano avait promis de ne pas étendre. Epargné de l'impossibilité de diffuser des nouvelles découvertes et en sachant que Del Ferro trouvait la solution avant Tartaglia, Cardano et Ferrari allaient rencontrer Annibale Della Nave, neuf de Del Ferro et son successeur à l'Université de Boulogne. Della Nave leur résiliait un occulte manuscrite avec la solution de l'équation, qui était la même trouvée par Tartaglia. En 1545, Cardano publiait la version de la résolution des équations de troisième degré accrédité à Ferrari.
Tartaglia, cependant, pensait d'être défraie ainsi qu'il commenéait une longue ébatailleé avec eux. Cardano et Ferrari conclurent in a Courte d'une église des Franciens à Milan avec beaucoup de gens et Ferrari était le champion de la défie. Donc, Tartaglia décidait de laisser Milan et fut mort avant de publier un traité des équations de troisième degré ainsi que les formules sont aujourd'hui reportées en aussitôt des livres comme les connues Formules de Cardano qui éclusaient les contributions de Del Ferro et de Tartaglia. Pour cette raison, Cardano est parfois renommée le voleur des formule, mais il était accusé d'ingénuités à cause de ne pas rebuter à soi-même la découverte en le traité. Parfois, il peut etre mieux que les appeler les formules de Del Ferro-Tartaglia-Cardano: trois auteurs pour une équation de troisième degré. Si on veut écrire en langage moderne, la solution de Cardano pour les équations cubiques de type $x^3+px=q$ est:
\[x=\sqrt[3]{\sqrt{\left( \frac{p}{3}\right)^3+\left( \frac{q}{2}\right)^2}+\frac{q}{2}}-\sqrt[3]{\sqrt{\left( \frac{p}{3}\right)^3+\left( \frac{q}{2}\right)^2-\frac{q}{2}}}\] et la formule pour l'équation cubique de type $x^3+px^2=q$ est: \[x=\sqrt[3]{\sqrt{\left( \frac{q}{2}\right)^2-\left( \frac{p}{3}\right)^{3}}+\frac{q}{2}}-\sqrt[3]{\sqrt{\left( \frac{q}{2}\right)^2+\left( \frac{p}{3}\right)^3-\frac{q}{2}}}.\]
Regardez en outre \cite{Chahal}, \cite{Cox} et \cite{Grant} pour une détaillé histoire des formules de Cardano. Le mathématicien que pour la première fois reconnait le besoin d'épandre les nombres aprés connus comme d'autres nombres était Rafael Bombelli (1526-1573), un mathématicien de Boulogne. Bombelli, en \textit{L'algèbre}, proposait de compléter les différents cas de résolution des équations de troisième degré, les cas irréductibles, c'est-à-dire que la formule de Cardano donnait la racine carrée d'un nombre négatif qui se représente quand $\left( \frac{p}{3}\right)^{3}+\left( \frac{q}{2}\right)^{2}<0$. En le premier livre de \textit{L'algèbre}, Bombelli examine les racines imaginaires des équations et revient sur la dénomination de ces nombres dits "complexes". En effet, il travaillait avec les radicaux également qu'ils étaient réels, une totale abstraction pour un mathématicien de cette période. Bombelli, toutefois, comprenait le mérite d'introduire en les Mathématiques les nombres complexes et les résultats des régles des calculassions aussi bien d'avoir complété la théorie des équations de troisième degré en dépoilant tous les cas qui seront présentés tandis que Cardano et Ferrari ne développaient pas une théorie complète. Pour en savoir plus sur les travaux de Bombelli, regardez aussi \cite{Gavagna,Gavagna2}.

\subsection{La résolution ordinaire de l'équation cubique}

Initialement, la résolution de l'équation de troisième degré $x^3+ax^2+bx+c=0$ (le coefficient de $x^3$ est considéré égale à 1) comporte la substitution $x=y-\frac{a}{3}$ sans le terme quadratique: 
\[ \left(y-\frac{a}{3}\right)^3+a\left(y-\frac{a}{3}\right)^2+b\left(y-\frac{a}{3}\right)+c=0 \]
\[ y^3-ay^2+\frac{\left(a^{2}y\right)}{3}-\frac{a^3}{27}+ay^2-\frac{2}{3}a^2 y+\frac{a^3}{9}+by-\frac{ab}{3}+c=0\]
\[ y^3-\frac{a^2 y}{3}+by+\frac{2a^3}{27}-\frac{ab}{3}+c=0 \]
\[ y^3+\left(b-\frac{a^2}{3}\right)y+\frac{2a^3}{27}-\frac{ab}{3}+c=0 \]
\[ y^3+py+q=0 \]
où
\vspace{5pt}
\begin{equation*}
\begin{cases}
p=b-\frac{a^2}{3}= \\
c-\frac{ab}{3}+\frac{2a^3}{27}.
\end{cases}
\end{equation*}
Nous allons trouver la solution $y$ comme la somme de deux nombres $u$ et $v$: $y=u+v$. Par substitution, nous avons: 
$0=y^3+py+q=\left(u+v\right)^3+p(u+v)+q=u^3+3u^{2}v+3uv^2+v^3+p(u+v)+q=((u^3+v^3+q))+(u+v)((3uv+p))$.
Cette expression est une équation en deux variables. Nous fixons par exemple u et nous allons obtenir une équation de troisième degré en $v$ qui est équivalent au problème initial. Nous remarquons qu'on peut trouver une combinaison de $u$ et $v$ qui sont calculé par moyens élémentaires. On suppose que $u$ et $v$ enlient les termes entre parenthèses, $u^3+v^3+q=0$ et $3uv+p=0$, c'est-à-dire:
\begin{equation*}
\begin{cases} 
u^3+v^3=-q \\
uv=-\frac{p}{3}.
\end{cases}
\end{equation*}
La résolution du système comporte à considérer la troisième puissance de la deuxième équation:
\begin{equation*}
\begin{cases} 
u^3+v^3=-q \\
u^{3}v^3=-\frac{p}{3}.
\end{cases}
\end{equation*}
Nous connaissons la somme et le produit de $u^3$ et $v^3$. Donc, l'équation $z^2+qz-\frac{p^3}{27}$ est résolue avec $z=-\frac{q}{2}\pm \sqrt{\frac{q^2}{4}+\frac{p^3}{27}}$. Nous avons $u^3=-\frac{q}{2}+\sqrt{\frac{q^2}{4}+\frac{p^3}{27}}$ et $v^3=-\frac{q}{2}-\sqrt{\frac{q^2}{4}+\frac{p^3}{27}}$, desquelles nous allons extraire les racines cubique et nous rappelons que $y=u+v$ et $y=\sqrt[3]{-\frac{q}{2}+\sqrt{\frac{q^2}{4}+\frac{p^3}{27}}}+\sqrt[3]{-\frac{q}{2}-\sqrt{\frac{q^2}{4}+\frac{p^3}{27}}}$ qui est la célèbre formule de la résolution des équations de troisième degré. Cette formule donne neuf valeurs de $x$ au moment où les radicaux cubiques possédent trois valeurs, qui sont obtenus d'un d'eux en multipliant l'un ou l'autre des deux racines de l'unité complexes cubiques $\epsilon_1=-\frac{1}{2}+i\frac{\sqrt{3}}{2}$ et $\epsilon_2=-\frac{1}{2}-i\frac{\sqrt{3}}{2}$. Entre les neuf valeurs, quand-même, il est nécessaire choisir ces qui satisfont la condition $uv=-\frac{p}{3}$. Les autres sont négligées. Toutefois, on va sélectionner une valeur pour le premier radical, choisir $v_1=-\frac{p}{3u_1}$ et nous avons comme racines $x_1=u_1-\frac{p}{3u_1}$, $x_2=\epsilon_{1}u_1+\epsilon_{2}v_2$ et $x_3=\epsilon_{2}u_1+\epsilon_{1}v_1$.
En 1861, le mathématicien anglais A. Cayley proposait une méthode qui conduisait à la formule de résolution et les trois racines. Par conséquence, les méthodes de résolution retrouvés fournirons la formule de Cayley sans prendre en considération la formule de Cardano. En effet, Cayley était le premier à investiguer le concept abstract d'un group et il commenéait la classification des groupes d'un donné ordre purement abstract. Nous nous referons à la théorie de Cayley, voyez par exemple \cite{Cayley}.

\section{Résolution des équations de quatrième degré}

\subsection{Equations de quatrième degré avant la découverte de la formule de résolution}

Les premières exemples d'équations de quatrième degré sont retrouvées dans les travaux arabes datés autour 1000. Un problème d'Abu'l-Wefa rapporté par Abu'l-Faradsh en Fihrist (ca. 987) est le suivant: "Retrouver la racine d'un cube, la quatrième puissance et les expressions composée de ces deux puissances", c'est-à-dire résoudre l'équation $x^4+px^3=q$. Ce travail contenant la résolution du problème posé a été perdu, mais nous recueillons qu'il peut être résolu en inter sécant l'hyperbole $y^2+axy+b=0$ avec la parabole $x^2-y=0$. Omar Khayyam, en l'\textit{algèbre}, résoudrait les équations de quatrième degré par moyen des méthodes géométriques déjà utilisés pour les équations de troisième degré. Par exemple, nous considérons le problème de construire un $ABCD$ tel que $AB=AD=BC=10$ et d'aire égale à $90$, qui conduit l'auteur à l'équation $(100-x^2)(100-x)^2=8100$. Omar Khayyam observait qu'une racine de ces équations coincide avec une des intersections de l'hyperbole $(100-x)y=90$ avec le cercle $x^2+y^2=100$. Au contraire, l'équation $x^4-2(x^2+200x)=9999$ a été retrouvée en le travail astronomique de l'indien Bhaskara (ca. 1150) qui retrouvait une de ses racines de la façon suivante. Avant tout, il ajoutait l'expression $4x^2+1$ aux deux membres obtenant $x^4-2x^2-400x+4x^4+1=4x^4+10000$ de laquelle $x^4+2x^2+1=4x^4+400x+10000\Leftrightarrow (x^2+1)^2=(2x+100)^2 \Leftrightarrow x^2+1=2x+100 \Leftrightarrow x^2-2x-99=0$ dont Bhaskara considérait seulement la solution $x=11$.
Plus tard, en la Summa de Luca Pacioli nous trouverons un problème présenté en la manière suivante: "En supposant que le circuit équatorial terrestre est de 20.400 miles et que par un point de vue de deux voyageurs qui partent ou deux en mouvements, pour prendre la route et aller de l'ouest à l'est en voyageant tous les jours majeurement en progression arithmétique, ainsi que le voyage du premier jour est seulement d'un mile, l'itinéraire du deuxième jour de deux miles, ce de troisième jour de trois miles, etc.; l'autre observateur de l'est à l'ouest avec des voyages suivants plus longs, comment dire les cubes des nombres, c'est-à-dire avec le voyage du premier jour d'un mile, le deuxième de huit miles, le troisième de vingt-sept, le quatrième de soixante-quatre..., on veut retrouver combien jours ont durés les voyages". Ce problème conduit à l'équation $\frac{1}{4}x^4+\frac{2}{4}x^3+\frac{3}{4}x^2+\frac{2}{4}x=20400$ attribuable à $x^4+2x^3+3x^2+2x=8100$. La méthode de Luca Pacioli pour la résolution de cette équation consiste en l'artifice suivant: il adjoint un aux deux membres, en obtenant $(x^2+x+1)^2=81601$ et, en considérant seulement la racine arithmétique du deuxième membre, comme il était le costume des temps, il venait de retrouver $x^2+x+1=\sqrt{81601}$, donc $x=-\frac{1}{2}+\sqrt{-\frac{3}{4}+\sqrt{81601}}$.
\subsection{La résolution de l'équation de quatrième degré par la méthode de Ferrari}
L'histoire qui concerne la découverte de la formule de la résolution de l'équation de quatrième degré commencé en 1535 avec une des questions de Master Zuanne de Tonini de Coi, en particulier en la question XX du travail Quesiti et différents inventions de Tartaglia: "... faire vingt à partir de trois parties proportionnelles continuent en sorte de proportions, multiplier deux termines pour la donnée de huit". Les trois quantités recherchées $x,y,z$ seront: 
\[x+y+z=20\]
\[x:y=y:z\]
\[xy=8.\]
Desquelles nous allons obtenir les équations:
\[x^4+8x^2+64=20x^3,\]
\[y^4+8y^2+64=160y.\]
\vspace{15pt}
Si la solution fut accomplie en trois ans, il était arrivé que Master Zuanne composait la même question à d'autres mathématiciens, bien lui Cardano. La défie qu'il proposait, recueillie par Ludovico Ferrari qui resalait victorieux, effectivement résolvant les questions alliées aux équations de quatrième degré avec une régle générale validée. Les résultats de Ferrari étaient démontrés en relation des cas numériques proposé par Master Zuanne, en l'Ars Magna de Cardano. Le problème quété de la résolution de l'équation algébrique de quatrième degré de type $x^4+ax^2+bx+c=0$ qui pas tout le monde réussit à résoudre. Ferrari, à travers la substitution $y=x-\frac{q}{4}$, transforme l'équation de quatrième degré la plus générale $y^4+py^3+qy^2+ry+s=0$ en l'équation $x^4+ax^2+bx+c=0$. Les équations de type $x^4+ax^2+bx+c=0$ étaient irrésoluble, mais Ludovico Ferrrari les résoudrait par moyen d'une procédure que Lagrange définit comme la plus simple de ces retrouvées d'aprés. L'artifice de L. Ferrari permit de réduire le problème à la résolution d'une équation de troisième degré. La procédure pour résoudre l'équation $x^4+6x^3+36=60x$ était reportée dans le travail de Cardano en la manière suivante:
\begin{enumerate}
\item sommer la quantité $12x^2$ aux deux membres de l'équation pour que le terme à gauche soit un carré: $x^4+6x^3+12x^2+36=60x+12x^2$
\item sommer aux deux membres de l'équation les termines du carré avec une nouvelle inconnue $y$ autant que le membre à gauche de l'équation reste un carré ainsi que le membre à droite, pour etre: $(x^2+y+6)^2=2(y+3)x^2+60x+y^2+12y$ qu'on peut écrire comme: \\
$(x^2+y+6)^2=\frac{(4(y+3)^{2}x^2+120x(y+3)+2(y+3)(y^2+12y))}{2(y+3)}$
\item choisir $y\in \mathbb{R}$ tel que le trinomyale du membre à droite soit un carré, par exemple: $2(y+3)(y^2+12y)=30^2=900$
\item développer les calculs pour être: $y^3+15y^2+36y=450$; cette équation est naturellement résoluble à travers les formules pour la résolution des équations de troisième degré
\item replacer les valeurs obtenues au quatrième passage avec y qui apparaissent en l'équation du deuxiéme et extraire la racine carré des deux membres
\item le résultat du cinquième passage est une équation de deuxième degré qui doit être résolue pour retrouver la valeur requise de $x$. \\
La solution de Ludovico Ferrari pour l'équation de quatrième degré est donnée en \cite{Anglin}. Nous voulons observer que la méthode appliquée au cas particulier que nous venons de considérer est en tout cas générale et c'est applicable à chaque équation de la même forme. La première exposition de la résolution d'une équation de quatrième degré, complète et exhaustive, fut retrouvée en l'\textit{Algebra} de R. Bombelli qui parfois vient d'être considéré erronément l'auteur de la formule de résolution. A cause de la limite de ne pas considérer le nombre négatifs en coefficients, Bombelli traite quarante-deux cas.
\end{enumerate}
\subsection{La résolution d'une équation de quatrième degré par la méthode d'Euler}
Une méthode pour résoudre une équation de quatrième degré qui suit celle de Tartaglia pour l'équation de troisième degré était proposée par Euler. La méthode est de la manière qui suit.
L'équation donnée est $y^4+py^2+qy+r=0$ et nous regardons la solution en la forme $y=u+v+w$. Nous prenons les carrés, $y^2=u^2+v^2+w^2+2(uv+uw+vw)$, nous transportons le quatrième termine du deuxième membre au premier et nous considérons les carrés: $y^4-2(u^2+v^2+w^2)y^2+(u^2+v^2+w^2)^2=4(u^{2}v^2+u^{2}w^2+v^{2}w^2)+8uvw(u+v+w)$, c'est-é-dire, $y^4-2(u^2+v^2+w^2)y^2-8uvw(u+v+w)+(u^2+v^2+w^2)^2-4(u^{2}v^2+u^{2}w^2+v^{2}w^2)=0$. Il s'agit d'une identité pour $y=u+v+w$, nous allons la comparer avec l'équation $y^4+py^2+qy+r=0$, et nous allons déduire que si $(u,v,w)$ est une solution du système
\begin{equation*}
\begin{cases}
u^2+v^2+w^2=-\frac{p}{2} \\
u^{2}v^2+u^{2}w^2+v^{2}w^2=\frac{p^2}{16}-\frac{r}{4} \\ 
uvw=-\frac{q}{8},
\end{cases}
\end{equation*}
l'addition $u+v+w$ est une racine de l'équation initiale de quatrième degré. Nous prenons le carré de la dernière équation ainsi que le système devient: 
\begin{equation*}
\begin{cases}
u^2+v^2+w^2=-\frac{p}{2} \\
u^{2}v^2+u^{2}w^2+v^{2}w^2=\frac{p^2}{16}-\frac{r}{4} \\ 
u^{2}v^{2}w^2=\frac{q^2}{64}.
\end{cases}
\end{equation*}
Nous connaissons la somme des carrés, la somme des produits deux-à-deux et les produits des nombres $u^2$, $v^2$, $w^2$ en utilisant les formules de F. Viéte, ce sont les solutions de l'équation de troisième degré $z^3+\frac{p}{2}z^2+(\frac{p^2}{16}-\frac{r}{4})z-\frac{q^2}{64}=0$ qui est nommé la résolvante de l'équation de quatrième degré $y^4+py^2+qy+r=0$. Une fois que nous avons obtenu les racines $u^2$, $v^2$, $w^2$, algébrique de la dernière équation, on va faire l'extraction des racines carrés, pour obtenir les valeurs de $u$, $v$ et $w$. Les signes des racines carrés doivent être choisi au fin de satisfaire la troisième condition du système. Soient $u_0,v_0,w_0$ les valeurs des inconnues $u,v,w$ telles que $u_{0}v_{0}w_0=-\frac{q}{8}$, alors $u_0+v_0+w_0$ est une racine de l'équation $y^4+py^2+qy+r=0$ et toutes les autres racines se calculent à travers le changement de signe à deux des nombres $u_0,v_0,w_0$. C'est l'unique moyen pour que le produit $uvw$ soit égal à $-\frac{q}{8}$. Ainsi, les racines de l'équation sont les suivantes: 
\[x_1=u_0+v_0+w_0\]
\[x_2=u_0-v_0-w_0\]
\[x_3=-u_0+v_0-w_0\]
\[x_4=-u_0-v_0+w_0.\]
L'équation $y^4+py^2+qy+r=0$ est supposée d'avoir $p,q$ et $r$ réels, $q\neq 0$ et de ne pas avoir des racines multiples et on va se référer à l'équation $z^3+\frac{p}{2}z^2+(\frac{p^2}{16}-\frac{r}{4})z-\frac{q^2}{64}=0$. Au moment où le résolvent a une racine réelle et deux racines complexes et conjuguées. Si le résolvant a trois racines réelles et positives, l'équation de quatrième degré a quatre solutions réelles et distinctes. Autrement, si le résolvant a trois racines réelles, une positive et deux négatives, alors l'équation initiale a quatre racines à deux-à-deux complexes et conjuguées. Finalement, le cas avec $q=0$ de l'équation $y^4+py^2+qy+r=0$ prend la forme $y^4+py^2+r=0$ qui s'appelle l'équation biquadratique dont la résolution est traéable à celle de l'équation de deuxième degré avec le replacement $y^2=x$. Pour connaitre majeurs détails sur la résolution de l'équation de quatrième degré avec la méthode d'Euler, regardez \cite{Nickalls}.
\section{La stratégie d'Evariste Galois}
Nous allons donner des résultats sur la résolution des équations de majeur degré. Pour la résolution des équations de cinquième degré, les mathématiciens cherchaient à la résolution de la formule par moyen des radicaux, similairement à ceux qui valent pour les équations jusqu'au quatrième degré, c'est-à-dire une formule qui soit l'expression des solutions par moyen d'opérations rationnelles et d'extractions des racines carrées de différents indexes. 
Les équations algébriques de cinquième et majeur degré furent étudiées, pendant deux siècles et plus, aprés Tartaglia et Cardano, sans que toutes les variétés des épreuves faites pour les résoudre des façons similaires à ces de deuxième ou troisième ou quatrième degré portaient les résultats espérés et les mathématiciens commençaient à penser que une formule pouvait ne pas exister.
En 1770, Waring, Vandermonde et Lagrange, de manière indépendante l'un de l'autre, publiaient leurs résultats: les deux premiers étaient de pensée pessimiste tant que Lagrange était de pensée optimiste; en effet, Euler trouvait une grande classe d'équations de cinquième degré qui pourrait être résolue par moyen de radicaux. 
Une réponse définitive est donnée surtout à travers la théorie de Galois qui implique la connaissance des extensions algébrique. Nous nous référons à \cite{Artin}, \cite{Lang} pour les détails. La résolution des équations de troisième et quatrième degré, en utilisant la théorie de Galois et quelques simple analyse de Fourier pour les groupes finis, a été accomplie en \cite{Janson}. Lire aussi \cite{Cox} pour les approfondissements.
\vspace{25pt}
\begin{theorem}[Théorème de Galois]
\label{thm1}
Soit $f(x)\in \mathbb{Q}[x]$ irréductible, avec les racines $r_1,...,r_n$. Soit $L=\mathbb{Q}(r_1,...,r_n)$. Si l'équation $f(x)=0$ est résoluble par moyen de radicaux, le groupe de Galois de $f$, $\textit{Gal}(L/\mathbb{Q})$, est un group résoluble.
\end{theorem}
\begin{proof}
Si $f$ est résoluble par moyen des radicaux, nous avons $\mathbb{Q}\subset R_1\subset...\subset R_t$ où chaque $R_i/R_{(i-1)}$, pour $1\leq i\leq t$, c'est une $p_i$-extension pure: $R_i=R_{(i-1)}a_{i}$ avec ${a_i}^{p_i}=\vartheta _{i}\in R_{(i-1)}$ et oé $L=\mathbb{Q}(r_1,...,r_m)\subset R_t$. On assume $R_t/\mathbb{Q}$ être normal. L'expression polynomiale $X^{p_i}-\vartheta _{i}\in R_{(i-1)}[X]R_t[X]$ a une racine $a_{i}$ en $R_t$, dont toutes ses racines sont en $R_t$ étant que $R_t/R_{(i-1)}$ est normal aussi que $\frac{R_t}{\mathbb{Q}}$. Les racines sont $\alpha_{i}a_i,...,{\alpha_{i}}^{p_i}a_i=a_i$ où $\alpha_i$ est une racine primitive $p$-ème de l'unité. $\alpha_i=\frac{\alpha_{i}(a_i)}{a_i}$ est dans $R_t$. Soient $n=p_{1}p_{2}...p_{t}$ et $\alpha$ une racine primitive de l'unité. Nous voulons modifier la relation initiale en la manière suivante: $\mathbb{Q}\subset \mathbb{Q}(\alpha)=K_0\subset K_1=R_1(\alpha)\subset ... \subset K_t=R_t(\alpha)=R_t (\alpha \in R_t)$.
\begin{description}
\item[(a)] Nous allons démontrer que $G_0=\textit{Gal}\left( \frac{K_t}{K_0}\right)$ est résoluble. Nous considérons $K_{i-1}\subset K_i\subset K_t$. Toutes les extensions sont normales: $\frac{K_t}{K_{i-1}}$ et $\frac{K_t}{K_i}$ comme $\frac{K_t}{\mathbb{Q}}$ est normale; $\frac{K_i}{K_{i-1}}$ comme $K_{i-1}=R_{i-1}(\alpha)$ contient une racine primitive p-ème de l'unité (si $\beta=\alpha^{(p_{1}...\hat{p}...p_t)}$ alors $\beta$ est une racine primitive p-ème de l'unité). Nous avons: $G_i=\textit{Gal}\left(\frac{K_t}{K_i}\right)\triangleleft \textit{Gal}\left(\frac{K_t}{K_{i-1}}\right)=G_{i-1}$ et $\frac{G_{i-1}}{G_i}\simeq \textit{Gal}\left(\frac{K_i}{K_{i-1}}\right)$. Mais, $\textit{Gal}\left(\frac{K_i}{K_{i-1}}\right)$ est cyclique d'ordre $p_i$. Alors, nous avons: $G_0=\textit{Gal}\left(\frac{K_t}{K_0}\right)\supset G_1=\textit{Gal}\left(\frac{K_t}{K_1}\right)\supset ... \supset G_{t-1}=\textit{Gal}\left(\frac{K_t}{K_{t-1}}\right)\supset G_t=\textit{Gal}\left(\frac{K_t}{K_t}\right)={1}$, où chaque $G_i$ est un sous-groupe normal de $G_{i-1}$ et chaque $\frac{G_{i-1}}{G_i}$ est cyclique d'ordre un nombre premier. Donc, $G_0$ est résoluble.
\item[(b)] $\textit{Gal}\left( \frac{K_t}{\mathbb{Q}}\right)$ est résoluble. Considérons $\mathbb{Q}\subset K_0=\mathbb{Q}\left( \alpha \right)K_t$. Toutes les extensions sont normales. Donc, $G_0\triangleleft G=\textit{Gal}\left( \frac{K_t}{\mathbb{Q}}\right)$ et le quotient est $\textit{Gal}\left( \frac{\mathbb{Q}_{0}\left( \alpha \right)}{\mathbb{Q}}\right)$. Ce dernier groupe est abélien, c'est-à-dire résoluble et $G$ est lui-même résoluble pour le point $(a)$.
\item[(c)] $\textit{Gal}\left( \frac{L}{\mathbb{Q}}\right)$ est résoluble. Nous considérons $\mathbb{Q}\subset L=\mathbb{Q}\left(r_1,...,r_n\right)\subset K_t$. Toutes les extensions sont normales et $\textit{Gal}\left(\frac{K_t}{L}\right)\triangleleft \textit{Gal}\left(\frac{K_t}{\mathbb{Q}}\right)=G$ et le quotient est $\textit{Gal}\left(\frac{L}{\mathbb{Q}}\right)$. Alors, $\textit{Gal}\left(\frac{L}{\mathbb{Q}}\right)$ est un quotient du groupe résoluble $\textit{Gal}\left(K_t/\mathbb{Q}\right)$ et $\textit{Gal}\left(\frac{L}{\mathbb{Q}}\right)$ est résoluble.
\end{description}
\end{proof}
Nous voudrions démontrer la deuxième partie du théorème de Galois: si $\textit{Gal}(f)$ est un groupe résoluble, alors $f(x)=0$ est résoluble par moyen de radicaux.
\begin{proposition}\label{prop1}
Soit $L/K$ une extension normale avec $|\textit{Gal}(L/K)|=p$ et $p$ un nombre premier. Si $K$ contient une racine primitive $p$-ème de l'unité alors $L=K(\beta)$ avec $\beta ^{p}\in K$ c'est-à-dire $L/K$ est une extension $p$-pure.
\end{proposition}
\begin{proof}
$L=K(v)$ et $v\notin K$ pour le théorème de l'élément primitif. Soit $\alpha \in K$ une racine primitive $p$-ème de l'unité $(\alpha \neq 1)$. Nous voulons remarquer que $K$ contient toutes les racines primitive $p$-ème que ce sont $\alpha,\alpha ^2,...,\alpha ^p=1$. L'hypothèse que $|\textit{Gal}(L/K)|=p$ implique que $\textit{Gal}(L/K)$ soit cyclique et que chaque élément différent de l'identité soit un générateur: $\textit{Gal}(L/K)=\left\{ \sigma,\sigma ^2,...,\sigma ^p=Id\right\}$. Pour toute $\varsigma$, racine $p$-ème de l'unité, nous avons: $t(\varsigma)=v+\varsigma(v)+\varsigma ^{2}\sigma ^{2}(v)+...+\varsigma ^{p-1}\sigma ^{p-1}(v)$. 
\begin{description}
\item[(a)] Nous avons: $\sigma(t(\varsigma))=\varsigma ^{-1}t(\varsigma)$. En effet, $\sigma(t(\varsigma))=\sigma(v)+\varsigma \sigma ^{2}(v)+...+\varsigma ^{p-1}(v)$ et nous allons conclure que $\varsigma ^{-1}=\varsigma ^{p-1}$. \\
\item[(b)] $\exists \varsigma \neq 1$ avec $t(\varsigma)\neq 0$. Nous avons: \\
\end{description}
\vspace{5pt}
\begin{equation*}
\begin{cases}
v+\sigma(v)+\sigma ^{2}(v)+...+\sigma ^{p-1}(v)=t(1) \\
v+\alpha \sigma(v)+\alpha ^2 \sigma ^{2}(v)+...+\alpha ^{p-1}\sigma ^{p-1}(v)=t(\alpha) \\
... \\
v+\alpha ^{i}\sigma(v)+\alpha ^{2i}\sigma ^{2}(v)+...+\alpha ^{i(p-1)}\sigma ^{p-1}(v)=t(\alpha ^{i}) \\
... \\
v+\alpha ^{n-1}(v)+\alpha ^{2(n-1)}\sigma ^2(v)+...+\alpha ^{(p-1)^2}\sigma ^{p-1}(v)=t(\alpha ^{p-1}).
\end{cases}
\end{equation*}
Si $t(\alpha)=...=t(\alpha ^{p-1})=0$, alors $v=\frac{1}{p}t(1)$. Donc, $\sigma(v)=\frac{1}{p}\sigma(t(1))$ et $\sigma(v)=v$. Cela implique que $\sigma ^{i}(v)=v$ et que $v\in K$, le champ fixé de $\textit{Gal}(L/K)$, comme $L/K$ est normal; mais il s'agit d'une contradiction avec $L=K(v)$.
Soit $\beta=t(\varsigma)$, $\varsigma \neq 1$ et $t(\varsigma)\neq 0$. On peut écrire $\sigma(\beta ^{p})=\varsigma ^{-1}\beta$, comme au point $(a)$, donc $\beta \notin K$ parce que $\varsigma \neq 1$ et $\sigma(\beta)\neq \beta$. En outre, $\sigma(\beta ^{p})=\sigma(\beta)^{p}=(\varsigma ^{-1}\beta)^p=\beta ^{p}$. Donc, $\sigma ^{i}(\beta ^{p})=\beta ^{p}$ et il implique que $\beta ^{p} \in K$. Nous avons $K\subset K(\beta)\subset L$ et $p=\dim_{K}(L)=\dim_{K_(\beta)}L \cdot \dim_{K}K(\beta)$, mais $\dim_{K}K(\beta)>1$ avec $\beta \notin K$, $\dim_{K}K(\beta)=p$ et $K(\beta)=L$ nécessairement.
\end{proof}
Avant de démontrer un deuxième résultat de E. Galois, nous avons ce qui suit.
\begin{proposition}[Irrationalités naturelles.]
Soit $f(x)\in K[x]$ une polynomiale irréductible. Nous allons dénoter avec $M=K(r_1,...,r_n)$ le champ de séparation de $f$ dédans $K$. Soit $L/K$ une extension et $N=L(r_1,...,r_n)$ le champ de séparation de $f$ dédans $L$. Alors, $\textit{Gal}(N/L)$ est isomorphique au sous-groupe $\textit{Gal}(M/ML)$ de $\textit{Gal}(M/K)$.
\end{proposition}
\begin{proof}
Nous avons le diagramme suivant:
\vspace{300pt}

\begin{figure}
\includegraphics[scale=0.7]{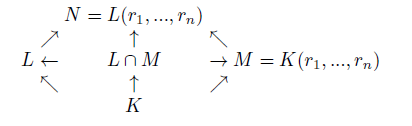}
\end{figure}
Nous allons définir $:Aut_{L}(N)Aut_{L\cap M}(M),\sigma \mapsto \sigma_{|M}\in Aut(M)$. Nous allons démontrer que $\psi$ est bien définie. Clairement, la restriction de $\sigma$ est un K-morphisme $\sigma_{|M}:M\rightarrow N\subset \Omega$. $M/K$ est normal, $\sigma_{|M}(M)=M$ et $\sigma_{|M}\in Aut(M)$. En outre, $\sigma_{|_{M\cap L}}=Id$ et $\psi$ est bien définie. Ce n'est pas difficile à démontrer que $\psi$ est un morphisme de groupes. 
\begin{description} 
\item[a] $\psi$ est injectif: si $\psi(\sigma)=Id$, alors $\sigma_{|M}:K(r_1,...,r_n )\rightarrow M=K(r_1,...,r_n)$ est telle que $\sigma(r_i)=r_i$ impliquant $\sigma=Id_L$.
\item[b] $\psi$ est surjectif: soit $G^{'}=(\textit{Gal}(N/L)), G^{'}$ est un sous-groupe de $\textit{Gal}(M/M\cap L)$. Soit $F=M^(G^{'})$ le sous-champ de $M$ fixé par $G$. Nous avons: $M\cap L\subset F$. Si $x\in M$, $x\notin M\cap L$ et $L$ est le champ fixé de $\textit{Gal}(N/L)$ parce que $N/L$ est normal, il existe $\tau \in Aut_{L}(N)$ tel que $\tau(x)\neq x$. Alors, $\psi(\tau(x))\neq x$ et on va conclure que $F=M\cap L$ et $G^{'}=\textit{Gal}(\frac{M}{M\cap L})$ comme $M/K$ est normal par  la correspondance de Galois.
\end{description}
\end{proof}
On va démontrer le deuxième théorème de Galois.
\begin{theorem}
\label{thm2}
Soit $f(x)\in \mathbb{Q}[x]$ une polynomiale irréductible. Soit $L=\mathbb{Q}(r_1,...,r_n)$ le champ de séparation de $f$ dédans $\mathbb{Q}$. Si \textit{Gal}$\left( \frac{L}{\mathbb{Q}}\right)$ est résoluble alors l'équation $f(x)=0$ est résoluble par moyen de radicaux.
\end{theorem}
\begin{proof}
Soit $k=|\textit{Gal}\frac{L}{\mathbb{Q}}|$ et $\epsilon$ une racine primitive $k$-ème de l'unité. Nous considérons $\frac{\mathbb{Q}\left( \epsilon \right)}{\mathbb{Q}}$: c'est une pure extension donc une extension radicale. Soit $N=\mathbb{Q}\left( \epsilon,r_1,...,r_n \right)$ et, en suivant la \ref{prop1}, $G=\textit{Gal}\frac{N}{\mathbb{Q}\left( \epsilon \right)}$ est isomorphe au sous-groupe de \textit{Gal}$\left( \frac{L}{\mathbb{Q}} \right)$ et $G$ est résoluble. Nous avons: 
\[G=G_0\supset G_1\supset ...\supset G_r={1}\]
avec $G_{i+1}\triangleleft G_i$ et $\frac{G_i}{G_{i+1}}$ est cyclique d'ordre un nombre premier. Nous avons, à travers la correspondance de Galois:
\[ \mathbb{Q}\left( \epsilon \right)=K_0\subset K_1\subset ... \subset K_r=N\]
où $\frac{K_{i+1}}{K_i}$ est normal avec \textit{Gal}$\left( \frac{K_{i+1}}{K_i}\right) \simeq \textit{Gal}\left( \frac{N}{K_{i}}\right)\textit{Gal}\left( \frac{N}{K_{i+1}}\right) \simeq \frac{G_i}{G_{i+1}}$. \textit{Gal}$\left( \frac{K_{i+1}}{K_i}\right)$ est cyclique d'ordre un nombre premier. Soit $d=|G|$ et $d|k$ (pour le théorème de Lagrange). $p=\textit{Gal}\left( \frac{K_{i+1}}{K_i}\right)$ divise $k$, comme $|G_i|,|G_{i+1}|$ divisent $d$, on écrit $pt=k$. Nous avons $\epsilon ^t \in K_0\subset K_i$ et $K_i$ contient une racine primitive $p$-ème de l'unité (c'est-à-dire $\beta=\epsilon ^t$). Cela implique que $K_{i+1}=K_{i}a_{i}$ avec $a_{i}^{p}\in K_i$. $\frac{N}{\mathbb{Q}\left( \epsilon \right)}$ est radicale ainsi que $\frac{N}{\mathbb{Q}}$ comme $\frac{\mathbb{Q}\left( \epsilon \right)}{\mathbb{Q}}$ est radical. Alors $r_1,...,r_n \in N$ et l'équation est résoluble par moyen des radicaux.
\end{proof}
Nous allons mettre ensemble le théorème \ref{thm1} et le théorème \ref{thm2} pour avoir le résultat suivant.
\begin{theorem}[Galois]
Soit $f(x)\in \mathbb{Q}[x]$ une polynomial irréductible. Soit $L=\mathbb{Q}(r_1,...,r_n)$ le champ de séparation de $f$ dédans $\mathbb{Q}$. L'équation $f(x)=0$ est résoluble par moyen des radicaux si et seulement si \textit{Gal}$\left( \frac{L}{\mathbb{Q}}\right)$ est un groupe résoluble.
\end{theorem}
Pour conclure le problème de la résolution des équations algébrique par moyen des radicaux nous avons besoin des résultats suivants.
\begin{theorem}[Abel]
L'équation générale de cinquième degré avec coefficients en $\mathbb{Q}$ n'est pas résoluble par moyen des radicaux.
\end{theorem}
\begin{proof}
Il est suffisant de retrouver une équation de cinquième degré avec le groupe de Galois non résoluble. Par exemple, on prend $S_5$ qui est non résoluble et si $f(x)=x^5-4x+2$, \textit{Gal}$_{\mathbb{Q}}F=S_5$.
\end{proof}
\begin{theorem}[Abel]
Soit $f(x)\in \mathbb{Q}[x]$ irréductible telle que \textit{Gal}$_{\mathbb{Q}}(f)$ est abélien, alors $f(x)=0$ est résoluble par moyen des radicaux.
\end{theorem}
\begin{proof}
Clairement, chaque groupe de Galois est résoluble.
\end{proof}
Les groups commutatifs sont nommés abéliens en mémoire de ce résultat. \\
E. Galois écrivait la fin de la défie de retrouver les solutions des équations algébriques par moyen des radicaux et il recueillit ses idées dans le théorème suivant. En effet, il découvrait avant sa mort, en 1832, qu'il n'existait pas une formule résolvante pour une polynomiale générale de degré majeur de quatre. Toutefois, ils existent plusieurs méthodes pour approximer les racines de ces polynomiales.
\begin{theorem}[Galois]
Soit $f(x)\in \mathbb{Q}[x]$ une polynomial irréductible avec degré un nombre premier $p$. Soient $r_1,...,r_p$ de $f$ et $L=\mathbb{Q}(r_1,...,r_p)$ le champ de séparation. Alors, $f(x)=0$ est résoluble par moyen des radicaux si et seulement si $L=\mathbb{Q}(r_i,r_j)$, pour tous $i,j,i\neq j$ (cela signifie que, comme Galois disait, $f(x)=0$ est résoluble par moyen des radicaux si et seulement si toutes les racines sont des fonctions rationnelles de deux entre elles).
\end{theorem}
\section{Conclusion}
Nous avons souligné l'importance des résultats obtenues par algébristes italiens dans leurs chef d'oeuvres sur la résolution des équations algébriques de troisième degré qui portaient à leur résolution et à la découverte des nombres complexes qui ont donné l'harmonie en la théorie des équations algébriques. Nous avons reporté des méthodes classique pour la résolution des équations de quatrième degré. Finalement, on a fixé l'attention aux travaux d'E. Galois qui a donné une première application de la théorie des groupes dans un sens moderne, pour importer l'ancien problème regardant les équations algébriques, les méthodes et les techniques pour établir quand, comment et les cas d'applications d'une équation donnée qui puisse être résolue en utilisant les radicaux.
\vspace{50pt}
\section{Déclaration}
L'autrice affirme que des conflits d'intêret n'existent aucuns.
\vspace{50pt}
\renewcommand\refname{8. Références}

\end{document}